\documentclass[12pt]{amsart} 

\usepackage{fullpage}
\usepackage{url,amssymb,amsmath,amsfonts,amsthm,mathrsfs}

\usepackage[all]{xy} \SelectTips{cm}{10}

\usepackage{lipsum}

\usepackage{color}

\definecolor{webcolor}{rgb}{0.8,0,0.2}
\definecolor{webbrown}{rgb}{.6,0,0}
\usepackage[
        colorlinks,
        linkcolor=webbrown,  filecolor=webcolor,  citecolor=webbrown, 
        backref,
        pdfauthor={David Zywina}, 
       pdftitle={Rank one elliptic curves and rank stability},
]{hyperref}
\usepackage[alphabetic,backrefs,lite]{amsrefs} 

\numberwithin{equation}{section}

\newcommand{\FF}{\mathbb F}

\newcommand{\QQ}{\mathbb Q}

\newcommand{\ZZ}{\mathbb Z}

\newcommand{\OO}{\mathcal O}

\newcommand{\calI}{\mathcal I}

\newcommand{\calP}{\mathcal P}

\newcommand{\calE}{\mathcal E}

\newcommand{\p}{\mathfrak p}
\newcommand{\m}{\mathfrak m}
\newcommand{\n}{\mathfrak n}

\def\rank{{\operatorname{rank}}}

 \def\Cl{\operatorname{Cl}}

\def\Sel{\operatorname{Sel}} 

\def\Gal{\operatorname{Gal}}

\def\Frob{\operatorname{Frob}}

\newcommand{\q}{\mathfrak q}

\newcommand{\legendre}[2]{\genfrac{(}{)}{}{}{#1}{#2}}

\newcommand{\defi}[1]{\textsf{#1}} 

\newcommand\blank[1]{}

\def\bbar#1{\setbox0=\hbox{$#1$}\dimen0=.2\ht0 \kern\dimen0 
\overline{\kern-\dimen0 #1}}
 
\newcommand{\Kbar}{{\bbar{K}}}

\newtheorem{thm}{Theorem}[section]
\newtheorem{lemma}[thm]{Lemma}

\theoremstyle{definition}

\theoremstyle{remark}

\newenvironment{romanenum}{\hfill \begin{enumerate} }{\end{enumerate}}
\newenvironment{alphenum}{\hfill \begin{enumerate} }{\end{enumerate}}

\begin{document}

\title{Rank one elliptic curves and rank stability}

\author{David Zywina}
\address{Department of Mathematics, Cornell University, Ithaca, NY 14853, USA}
\email{zywina@math.cornell.edu}

\begin{abstract}
For any quadratic extension $L/K$ of number fields, we prove that there are infinitely many elliptic curves $E$ over $K$ so that the abelian groups $E(K)$ and $E(L)$ both have rank $1$.   In particular, there are infinitely many elliptic curves of rank $1$ over any number field.  This result generalizes theorems of Koymans-Pagano and Alp\"oge–Bhargava–Ho–Shnidman which were used to independently show that Hilbert's  tenth problem over the ring of integers of any number field has a negative answer.  Our approach differs since we are obtaining our elliptic curves by specializing a nonisotrivial rank $1$ family of elliptic curves and we compute all the ranks involved.
\end{abstract}

\maketitle

\section{Introduction}

Let $K$ be any number field.  For an elliptic curve $E$ defined over $K$, the abelian group $E(K)$ consisting of the $K$-rational points of $E$ is finitely generated.   The \defi{rank} of $E$ is the rank of the abelian group $E(K)$.   

A \emph{minimalist conjecture} loosely predicts that if we order elliptic curves over $K$ by their height, the density of those curves with rank $0$ and the density of those curves with rank $1$ should both be $1/2$ (the idea being that a typical curve over $K$ should have the smallest possibly rank that is compatible with the parity conjecture).   In their work on Hilbert's tenth problem, Mazur and Rubin proved that there are elliptic curves over $K$ of rank $0$, cf.~\cite[Theorem 1.11]{MR2660452}.    Our first result guarantees the existence of elliptic curves over $K$ of rank $1$.

\begin{thm} \label{T:cor}
For any number field $K$, there are infinitely many elliptic curves $E$ over $K$, up to isomorphism, of rank $1$.
\end{thm}

Theorem~\ref{T:cor} was already known for $K=\QQ$, for example see \cite[Th\'eor\`eme~3.1]{MR870738} or \cite{MR3237733}, but these methods do not readily extend to a general number field.   One also knew that Theorem~\ref{T:cor} held under various conjectures;  for example \cite[Corollary 1.12]{MR2660452} implies the theorem assuming the finiteness of the $2$-part of the Tate--Shafarevich group for elliptic curves over number fields.  Our proof builds off of the ideas introduced  in \cite{Zyw25a} and \cite{Zyw25b}; in the second paper we proved that there are infinitely many elliptic curves over $\QQ$ of rank $2$.   An independent proof of Theorem~\ref{T:cor} has been given by Koymans and Pagano \cite{KPnew} which extends their earlier techniques from \cite{KP};  they show that a ``generic'' elliptic curve over $K$ with full $2$-torsion has infinitely many quadratic twists of rank $1$.

Our main result is the stronger rank stability theorem.

\begin{thm} \label{T:main}
For any quadratic extension $L/K$ of number fields, there are infinitely many elliptic curves $E$ over $K$, up to isomorphism, for which $ \rank \,E(K)= \rank \,E(L)=1$.
\end{thm}

The interest in Theorem~\ref{T:main} is that it generalizes the following two recent results which both imply that Hilbert's tenth problem for the ring of integers of a number field has a negative answer.   

\begin{thm}[Koymans--Pagano \cite{KP}*{Theorem~2.4}] \label{T:KP}
Let $K$ be a number field with at least $32$ real places and define $L:=K(i)$.   Then there are infinitely many elliptic curves $E$ over $K$ such that $\rank\, E(K) = \rank\, E(L)>0$.
\end{thm}

\begin{thm}[Alp\"oge–Bhargava–Ho–Shnidman \cite{ABHS}*{Theorem~1.1}] \label{T:ABHS}
For any quadratic extension $L/K$ of number fields, there are infinitely many abelian varieties $A$ over $K$ such that $\rank\, A(K) = \rank\, A(L)>0$.
\end{thm}

The theorems from \cite{KP} and \cite{ABHS} are both stated in terms of the existence of such an abelian variety but it is clear from their proofs that they can produce infinitely many.   Note that the existence is all that is needed for their application to Hilbert's tenth problem.

\subsection{Our curves}

Our approach to proving Theorem~\ref{T:cor} is to consider the elliptic curve $\calE$ over the function field $\QQ(T)$ defined by the Weierstrass equation
\[
y^2=x^3+4T(T+1)x^2 +2T(T+1)^2(T-1) x.
\]
The discriminant of this Weierstrass model is $2^9 T^3 (T+1)^7 (T-1)^2$.   The group $\calE(\QQ(T))$ contains a $2$-torsion point $(0,0)$ and a point $P:=(-2T(T+1),2T(T+1)^2)$ of infinite order.  

Consider any number field $K$. For each $t\in K-\{0,1,-1\}$, the specialization of $\calE$ at $t$ is the elliptic curve $\calE_t$ over $K$ defined by the equation $y^2=x^3+4t(t+1)x^2 +2t(t+1)^2(t-1) x$ with a point $P_t:=(-2t(t+1),2t(t+1)^2) \in \calE_t(K)$.  From Silverman \cite{MR703488} we know that the point $P_t$ will have infinite order for all but finitely many choice of $t$.   For such $t$, we have $\rank\, \calE_t(K) \geq 1$.  The challenge is to find examples for which equality holds.

Our strategy is to try and choose $t\in K-\{0,1,-1\}$ such that an explicit $2$-descent will produce the desired upper bound $\rank \, \calE_t(K) \leq 1$.  To perform a $2$-descent we will need to fully understand the bad primes of $\calE_t$.  Taking $t=a/b$ with $a,b\in K$ such that $ab(a+b)(a-b) \neq 0$ and scaling the coordinates of the Weierstrass model appropriately, we find that $\calE_t/K$ can be given by the Weierstrass equation
\begin{align}\label{E:intro strategy weierstrass}
y^2=x^3+4a(a+b)x^2 +2a(a+b)^2(a-b) x
\end{align}
which has discriminant $\Delta=2^9 a^3 (a+b)^7 (a-b)^2$.  

\subsection{Technical result}

For a finite set $S$ of nonzero prime ideals of $\OO_K$, we will denote by $\OO_{K,S}$ the ring of $S$-integers in $K$.   The main task of this paper is to prove the following.

\begin{thm} \label{T:true}
Let $K$ be a number field and fix a $D\in K$ that is not a square.   There is a finite set $S$ of nonzero prime ideals of $\OO_K$ and a nonempty open subset $U_\p$ of $K_\p^2$ for each $\p\in S$ such that if $a$ and $b$ are elements of $\OO_{K,S}$ that satisfy the following conditions:
\begin{itemize}
\item
$a$, $a+b$ and $a-b$ generate distinct nonzero prime ideals of $\OO_{K,S}$,
\item
$(a,b)$ lies in $U_\p$ for all $\p\in S$,
\item
$0<a<b$ in $K_v$ for all real places $v$ of $K$,
\end{itemize}
then the elliptic curve over $K$ defined by (\ref{E:intro strategy weierstrass}) has rank $1$ and its quadratic twist by $D$ has rank $0$.
\end{thm}

The key idea of Theorem~\ref{T:true} is that the sets $U_\p$ are forcing very specific local conditions when performing the relevant $2$-descents.  The sets $S$ and $U_\p$ in our proof will be a little complicated but the general philosophy is that we are reverse engineering a successful $2$-descent.

In order to find $a$ and $b$ satisfying the conditions of Theorem~\ref{T:true}, we will use a result of Kai \cite[Proposition~13.2]{Kai} which generalizes work of Green, Tao and Ziegler to number fields.  With this result of Kai, we will see that Theorems~\ref{T:cor} and \ref{T:main} are easy consequences of Theorem~\ref{T:true}.   The work of Kai also shows up in the proofs of Theorem~\ref{T:KP} and \ref{T:ABHS}.  Since we are interested in primes in a $3$-term arithmetic progression $a-b,a,a+b$, it seems likely that one could also use classical analytic number theory as in \cite[\S3]{ABHS} to prove the existence of $a$ and $b$.

\subsection{Notation}

Let $K$ be a number field and let $\OO_K$ be its ring of integers.   For each nonzero prime ideal $\p$ of $\OO_K$, let $v_\p$ be the discrete valuation on $K$ normalized so that $v_\p(K^\times)=\ZZ$.   Let $|\cdot |_\p$ be the absolute value on $K$ for which $|a|_\p=|\OO_K/\p|^{-v_\p(a)}$.   Let $K_\p$ be the $\p$-adic completion of $K$ and let $\OO_\p$ be its valuation ring.

For a finite set $S$ of nonzero prime ideals of $\OO_K$, let $\OO_{K,S}$ be the ring of $S$-integers, i.e., the ring of $a\in K$ for which $v_\p(a)\geq 0$ for all nonzero prime ideals $\p\notin S$ of $\OO_K$.

For a place $v$ of $K$, we denote by $K_v$ the completion of $K$ at $v$.   When $v$ is a finite place of $K$, we will frequently switch between $v$ and the corresponding prime ideal $\p$ of $\OO_K$.

\section{$2$-descent background}
Fix a number field $K$. In this section we recall the Selmer group associated to a degree $2$ isogeny of an elliptic curve over $K$.   We also compute the local relations that will be needed for our Selmer group calculations.

\subsection{$2$-descent background} \label{SS:2-descent}
We first recall some background on descent via a $2$-isogeny.  See \cite[\S X.4]{Silverman}, and in particular \cite[\S X.4 Example 4.8]{Silverman}, for details.

Let $E$ be an elliptic curve over a number field $K$ given by a Weierstrass equation $y^2=x^3+\alpha x^2+\beta x$ with $\alpha,\beta\in K$.   Let $E'/K$ be the elliptic curve defined by $y^2=x^3+\alpha'x^2+\beta'x$, where $\alpha':=-2\alpha$ and $\beta':=\alpha^2-4\beta$.  There is a degree $2$ isogeny $\phi\colon E\to E'$ given by $\phi(x,y)=(y^2/x^2, y(\beta-x^2)/x^2)$.  Let $\hat{\phi}\colon E'\to E$ be the dual isogeny of $\phi$.  The kernel of $\phi$ and the kernel of $\hat{\phi}$ are generated by the $2$-torsion point $(0,0)$ of $E$ and $E'$, respectively.

Set $\Gal_K:=\Gal(\Kbar/K)$.  Starting with the short exact sequence $0\to \ker \phi \to E \xrightarrow{\phi} E'\to 0$ and taking Galois cohomology yields an exact sequence
\[
0 \to \ker \phi \to E(K)\xrightarrow{\phi} E'(K) \xrightarrow{\delta} H^1(\Gal_K, \ker \phi).
\]
Since $\ker \phi$ and $\{\pm 1\}$ are isomorphic $\Gal_K$-modules, we have a natural isomorphism
\begin{align} \label{E:H1 isom}
H^1(\Gal_K, \ker \phi )\xrightarrow{\sim} H^1(\Gal_K, \{\pm 1\}) \xrightarrow{\sim} K^\times/(K^\times)^2.
\end{align}
Using the isomorphism (\ref{E:H1 isom}), we may view $\delta$ as a homomorphism 
\[
\delta\colon E'(K)\to  K^\times/(K^\times)^2
\]   
whose kernel is $\phi(E(K))$.
For any point $(x,y)\in E'(K)-\{0,(0,0)\}$, we have $\delta((x,y))=x\cdot (K^\times)^2$. We also have $\delta(0)=1$ and $\delta((0,0))=\beta' \cdot (K^\times)^2$.   

For each place $v$ of $K$, a similar construction defines a homomorphism
\[
\delta_v \colon E'(K_v) \to K_v^\times/(K_v^\times)^2
\]
with kernel $\phi(E(K_v))$.   We can identify the \defi{$\phi$-Selmer group} of $E/K$, which we denote by $\Sel_\phi(E/K)$, with the subgroup of $K^\times/(K^\times)^2$ consisting of those square classes whose image in $K_v^\times/(K_v^\times)^2$ lies in $\operatorname{Im}(\delta_v)$ for all places $v$ of $K$.  The image of $\delta$ lies in $\Sel_\phi(E/K)$ and hence we have an injective homomorphism
\[
E'(K)/\phi(E(K)) \hookrightarrow \Sel_\phi(E/K).
\]
Similar, we have an injective homomorphism $E(K)/\hat\phi(E'(K)) \hookrightarrow \Sel_{\hat\phi}(E'/K)$.  These groups will help us understand the rank of $E$ via the exact sequence
\begin{align} \label{E:exact MW}
0\to \frac{\langle (0,0) \rangle}{\phi(E(K)[2])} \to \frac{E'(K)}{\phi(E(K))} \xrightarrow{\hat\phi} \frac{E(K)}{2E(K)} \to \frac{E(K)}{\hat{\phi}(E'(K))} \to 0.
\end{align}

The following lemma links the cardinalities of $\Sel_\phi(E/K)$ and $\Sel_{\hat\phi}(E'/K)$.

\begin{lemma} \label{L:Selmer ratio}
\begin{romanenum}
\item \label{L:Selmer ratio i}
We have $|\Sel_\phi(E/K)|/|\Sel_{\hat\phi}(E'/K)| = \prod_v \tfrac{1}{2} |\operatorname{Im}(\delta_v)|$, where the product is over the places $v$ of $K$.
\item \label{L:Selmer ratio ii}
Let $\p$ be a nonzero prime ideal of $\OO_K$ that does not divide $2$.   Then 
\[
\tfrac{1}{2} |\operatorname{Im}(\delta_\p)| = c_\p(E')/c_\p(E),
\] 
where $c_\p(E)$ and $c_\p(E')$ are the Tamagawa numbers of $E$ and $E'$, respectively, at $\p$. 
\item \label{L:Selmer ratio iii}
If $v$ is a real place of $K$ and $\beta<0$ in $K_v$, then $\operatorname{Im}(\delta_{v})=1$.
\end{romanenum}
\end{lemma}
\begin{proof}
For each place $v$ of $K$, let $\phi_v\colon E(K_v)\to E'(K_v)$ be the homomorphism obtained from the isogeny $\phi$.   Define $T(E/E'):=\prod_{v} |\operatorname{coker}(\phi_v)|/|\ker(\phi_v)|$, where the product is over the places $v$ of $K$.   From \cite[Lemma~3.1 and equation (3.6)]{MR179169}, we find that all but finitely many terms of the product $T(E/E')$ are $1$ and hence $T(E/E')$ is well-defined.   Note that our product is the inverse of Cassels' but our notation is compatible by equation (1.23) of \cite{MR179169}.  By Theorem~1.1 of \cite{MR179169}, we have
\[
T(E/E')= \frac{|\Sel_\phi(E/K)|}{|\Sel_{\hat\phi}(E'/K)|} \cdot \frac{|\ker(\hat{\phi})(K)|}{|\ker({\phi})(K)|} = \frac{|\Sel_\phi(E/K)|}{|\Sel_{\hat\phi}(E'/K)|}.
\]
From our two expressions for $T(E/E')$, we find that $|\Sel_\phi(E/K)|/|\Sel_{\hat\phi}(E'/K)| = \prod_v \tau_v$, where $\tau_v:=|\operatorname{coker}(\phi_v)|/|\ker(\phi_v)|$.   We have $\tau_v=\tfrac{1}{2} |\operatorname{Im}(\delta_v)|$ since $\ker(\phi_v)=\langle (0,0) \rangle$ and $\operatorname{Im}(\delta_v) \cong E'(K_v)/\phi(E(K_v))$.  This completes the proof of (\ref{L:Selmer ratio i}).

Fix a nonzero prime ideal $\p\nmid 2$ of $\OO_K$.  For a nonsingular Weierstrass model $y^2+a_1xy+a_3y=x^3+a_2x^2+a_4x+a_6$ over $K$, we have a corresponding invariant differential $dx/(2y+a_1x+a_3)$ on the elliptic curve defined by this model.    Let $\omega$ be an invariant differential on $E$ arising from a Weierstrass model of $E$ that is minimal at $\p$.  Let $\omega'$ be an invariant differential on $E'$ arising from a Weierstrass model of $E'$ that is minimal at $\p$.   There is a unique $\alpha\in K^\times$ such that $\phi^*\omega'=\alpha \omega$.  By \cite[Lemma~4.2]{MR3324930}, we have
\[
\tfrac{1}{2} |\operatorname{Im}(\delta_\p)| = \frac{|\operatorname{coker}(\phi_\p)|}{|\ker(\phi_\p)|} = |\alpha|_\p^{-1} \frac{c_\p(E')}{c_\p(E)}.
\]
Since $\p\nmid 2=\deg \phi$, \cite[Lemma~4.3]{MR3324930} implies that $|\alpha|_\p^{-1}=1$ which proves (\ref{L:Selmer ratio ii}).

Finally fix a real place $v$ of $K$ for which $\beta<0$.   We have $\tau_v=1/2$ by \cite[Proposition~7.6]{MR3324930}.  Therefore, $|\operatorname{Im}(\delta_v)| = 2\tau_v=1$ and (\ref{L:Selmer ratio iii}) follows.
\end{proof}

\subsection{Our curves and local conditions} \label{SS:our curves 1}

Fix $a,b\in K$ such that $ab(a+b)(a-b)\neq 0$.  Take any $d\in K^\times$ and let $E_d$ be the elliptic curve over $K$ defined by the Weierstrass equation
\begin{align} \label{E:main}
y^2=x^3+4a(a+b) d\, x^2 +2a(a+b)^2(a-b)  d^2\, x;
\end{align}
its discriminant is $\Delta_d :=2^9 a^3 (a+b)^7 (a-b)^2 \cdot d^6 \neq 0$.   From \S\ref{SS:2-descent}, there is a degree $2$ isogeny $\phi_d \colon E_d\to E'_d$ with $\phi_d((0,0))=0$, where $E'_d$ is the elliptic curve over $K$ defined by the Weierstrass equation
\begin{align}\label{E:main2}
y^2=x^3-8a(a+b) d \, x^2 +8a(a+b)^3  d^2 \, x.
\end{align}
The discriminant of the Weierstrass model (\ref{E:main2}) is $\Delta'_d:=2^{15} a^3 (a+b)^8(a-b)\cdot d^6$.  As in \S\ref{SS:2-descent}, we have a homomorphism
\[
\delta_{d,v}\colon E'_d(K_v) \to K_v^\times/(K_v^\times)^2 
\]
with kernel $\phi(E_d(K_v))$ for every place $v$ of $K$.   

We now describe the image of $\delta_{d,v}$ in various cases that will arise in our proof of the main theorems.   For each nonzero prime ideal $\p$ of $\OO_K$, we have an inclusion $\OO_\p^\times/(\OO_\p^\times)^2 \hookrightarrow K_\p^\times/(K_\p^\times)^2$.

\begin{lemma} \label{L:Selmer}
Let $\p$ be a prime ideal of $\OO_K$ with $v_\p(2)=v_\p(d)=0$.  Define \[
m:=(v_\p(a), v_\p(a+b), v_\p(a-b) ).\]
\begin{romanenum}
\item \label{L:Selmer i}
If $m=(0,0,0)$, then $\operatorname{Im}(\delta_{d,\p}) = \OO_\p^\times/(\OO_\p^\times)^2$ and hence $|\operatorname{Im}(\delta_{d,\p})|=2$.

\item \label{L:Selmer ii}
If $m=(0,0,1)$, then $|\operatorname{Im}(\delta_{d,\p})|=1$.

\item \label{L:Selmer iii}
If $m=(1,0,0)$, then $|\operatorname{Im}(\delta_{d,\p})|=2$  and $\operatorname{Im}(\delta_{d,\p}) \neq \OO_\p^\times/(\OO_\p^\times)^2$.  

\item \label{L:Selmer iv} 
If $m=(0,1,0)$ and $d$ is a square in $K_\p$, then $|\operatorname{Im}(\delta_{d,\p})|=2$. 

\item \label{L:Selmer v}
If $m=(0,1,0)$ and $d$ is not a square in $K_\p$, then $|\operatorname{Im}(\delta_{d,\p})|=4$. 

\item \label{L:Selmer vi}
If $m=(-1,1,-1)$ and $-2a(a+b)d$ is a square in $K_\p$, then $|\operatorname{Im}(\delta_{d,\p})| = 4$.

\item \label{L:Selmer vii}
If $m=(-1,1,-1)$ and $-2a(a+b)d$ is not a square in $K_\p$, then $|\operatorname{Im}(\delta_{d,\p})| = 2$.

\item \label{L:Selmer viii} 
If $m=(0,2,0)$ and $-2a(a+b)d$ is a square in $K_\p$, then $|\operatorname{Im}(\delta_{d,\p})|=4$. 

\item \label{L:Selmer ix}
If $m=(0,2,0)$ and $-2a(a+b)d$ is not a square in $K_\p$, then $|\operatorname{Im}(\delta_{d,\p})|=2$. 
\end{romanenum}
\end{lemma}
\begin{proof}
By Lemma~\ref{L:Selmer ratio}(\ref{L:Selmer ratio ii}), we have $|\operatorname{Im}(\delta_{d,\p})| = 2c_\p(E'_d)/c_\p(E_d)$.   We now compute these Tamagawa numbers by using Tate's algorithm, cf.~\cite[Algorithm 9.4]{SilvermanII}.   The applications of Tate's algorithm below are straightforward; in each case, we give a specific Weierstrass model that is minimal at $\p$ for which the algorithm can be applied directly as in \cite[Algorithm 9.4]{SilvermanII} without having to change the model.  

First consider the case $m=(0,0,0)$.   Since $E_d$ and $E'_d$ have good reduction at $\p$, we have $|\operatorname{Im}(\delta_{d,\p})| = 2c_\p(E'_d)/c_\p(E_d)=2\cdot 1/1=2$.  We have $\delta_{d,\p}((0,0))= 2a(a+b) \cdot (K_\p^\times)^2$ and $2a(a+b)\in \OO_\p^\times$.   Now take any point $(x,y) \in E'_d(K_\p)-\{0,(0,0)\}$.   Since $\delta_{d,\p}((x,y))=x\cdot (K_\p^\times)^2$, to complete the proof of (\ref{L:Selmer i}) it suffices to show that $v_\p(x)$ is even.  Since $v_\p(2)=v_\p(a)=v_\p(a+b)=0$, from (\ref{E:main2}) we find that $2v_\p(y)=v_\p(y^2)=v_\p(x)$ if $v_\p(x)>0$ and $2v_\p(y)=v_\p(y^2)=3v_\p(x)$ if $v_\p(x)<0$.   Therefore, $v_\p(x)$ must be even.

Consider the case $m=(0,0,1)$.   We have $v_\p(\Delta_d)=2$ and $v_\p(\Delta'_d)=1$. Applying Tate's algorithm with (\ref{E:main}) we find that $E_d$ has Kodaira symbol $\operatorname{I}_2$ at $\p$ and hence $c_\p(E_d)=2$.  Replacing $x$ by $x+4a(a+b)d$ in the model (\ref{E:main2}), we obtain another Weierstrass model
\begin{align*} 
y^2=x^3+4a(a+b)dx^2-8a(a+b)^2(a-b)d^2x -32a^2(a+b)^3(a-b)d^3
\end{align*}
for $E'_d$.  Applying Tate's algorithm to this model, we find that $E'_d$ has Kodaira symbol $\operatorname{I}_1$ at $\p$ and hence $c_\p(E'_d)=1$.  Therefore, $|\operatorname{Im}(\delta_{d,\p})| = 2 \cdot 1/2=1$.

Consider the case $m=(1,0,0)$.  Note that $a$ is a uniformizer of $K_\p$.  Applying Tate's algorithm with (\ref{E:main}) and (\ref{E:main2}), we find that $E_d$ and $E'_d$ both have Kodaira symbol $\operatorname{III}$ at $\p$ and hence $c_\p(E_d)=c_\p(E'_d)=2$.  Therefore, $|\operatorname{Im}(\delta_{d,\p})| = 2 \cdot 2/2=2$.  We have $\operatorname{Im}(\delta_{d,\p}) \neq \OO_\p^\times/(\OO_\p^\times)^2$ since $\delta_{d,\p}((0,0))=2a(a+b)\cdot (K_\p^\times)^2$ and $v_\p(2a(a+b))=1$.

Consider the case $m=(0,1,0)$.  Note that $a+b$ is a uniformizer of $K_\p$.   Replacing $x$ by $x-2a(a+b)d$ in the model (\ref{E:main}), we obtain another Weierstrass model
\begin{align}\label{E:main-redux}
y^2=x^3-2a(a+b) d x^2-2a(a+b)^3 d^2 x +4a^2(a+b)^4 d^3
\end{align}
for $E_d$.  Applying Tate's algorithm to this model and using $v_\p(\Delta_d)=7$, we find that $E_d$ has Kodaira symbol $\operatorname{I}_1^*$ at $\p$ with $c_\p(E_d)=4$ if $d$ is a square in $K_\p$ and $c_\p(E_d)=2$ otherwise.   Applying Tate's algorithm with (\ref{E:main2}) and using $v_\p(\Delta_{d}')=8$, we find that $E'_d$ has Kodaira symbol $\operatorname{I}_2^*$ at $\p$ with $c_\p(E'_d)=4$.    Therefore, $|\operatorname{Im}(\delta_{d,\p})|$ is equal to $2 \cdot 4/4=2$ if $d$ is a square in $K_\p$ and is equal to $2\cdot 4/2=4$ otherwise.  

Consider the case where $m=(-1,1,-1)$.    Applying Tate's algorithm to the model (\ref{E:main-redux}) for $E_d$, and using $v_\p(\Delta)=2$, $v_\p(2a(a+b))=0$ and $v_\p(2a(a+b)^3)=v_\p(4a^2(a+b)^4)=2$, we find that $E_d$ has Kodaira symbol $\operatorname{I}_2$ at $\p$ and hence $c_\p(E_d)=2$.  Using that $v_\p(\Delta'_d)=4$, $v_\p(8a(a+b))=0$ and $v_\p(8a(a+b)^3)=2$, Tate's algorithm applied to (\ref{E:main2}) shows that $E'_d$ has Kodaira symbol $\operatorname{I}_4$ at $\p$ and has split multiplication at $\p$ if and only if $-2a(a+b)d$ is a square in $K_\p$.   Therefore, $c_\p(E'_d)=4$ if $-2a(a+b)d$ is a square in $K_\p$ and $c_\p(E'_d)=2$ otherwise.   Therefore, $|\operatorname{Im}(\delta_{d,\p})|$ is equal to $2 \cdot 4/2=4$ if $-2a(a+b)d$ is a square in $K_\p$ and is equal to $2\cdot 2/2=2 $ otherwise.   

Consider the case $m=(0,2,0)$.  Fix a uniformizer $\pi$ of $K_\p$.  Dividing both sides of (\ref{E:main-redux}) by $\pi^6$ and changing coordinates we obtain a model
\begin{align} \label{E:new E 020}
y^2=x^3-2a(\tfrac{a+b}{\pi^2}) d x^2-2a(\tfrac{a+b}{\pi^2})^3 \pi^2 d^2 x +4a^2(\tfrac{a+b}{\pi^2})^4 \pi^2d^3
\end{align}
for $E_d$.   The Weierstrass equation (\ref{E:new E 020}) has coefficients in $\OO_\p$ with $-2a(\tfrac{a+b}{\pi^2})d \in \OO_\p^\times$ and its discriminant is $2^9a^3(a+b)^7(a-b)^2d^6 /\pi^{12}$ which has $\p$-adic valuation $2$.  Tate's algorithm shows that $E_d$ has Kodaira symbol $\operatorname{I}_2$ at $\p$ and hence $c_\p(E)=2$.  Dividing both sides of (\ref{E:main2}) by $\pi^6$ and changing coordinates we obtain a model
\begin{align} \label{E:new2 E 020}
y^2=x^3-8a(\tfrac{a+b}{\pi^2}) d  x^2 +8a(\tfrac{a+b}{\pi^2})^3  \pi^2 d^2  x.\end{align}
for $E_d$.   The Weierstrass equation (\ref{E:new2 E 020}) has coefficients in $\OO_\p$ with $-8a(\tfrac{a+b}{\pi^2})d \in \OO_\p^\times$ and its discriminant is $2^{15}a^3(a+b)^8(a-b)d^6/\pi^{12}$ which has $\p$-adic valuation $4$.  Tate's algorithm shows that $E_d$ has Kodaira symbol $\operatorname{I}_4$ at $\p$ and has split multiplicative reduction at $\p$ if and only if $-2a(a+b)d$ is a square in $K_\p$.   Therefore, $c_\p(E')=4$ if $-2a(a+b)d$ is a square in $K_\p$ and $c_\p(E')=2$ otherwise.  Therefore, $|\operatorname{Im}(\delta_{d,\p})|$ is equal to $2 \cdot 4/2=4$ if $-2a(a+b)d$ is a square in $K_\p$ and is equal to $2\cdot 2/2=2 $ otherwise.   
\end{proof}

The following easy lemma will be used later to ensure that the local conditions we impose can be satisfied.

\begin{lemma} \label{L:local existence}
Let $\p\nmid 2$ be a nonzero prime ideal of $\OO_K$ and fix a uniformizer $\omega$ of $K_\p$.
\begin{romanenum}
\item \label{L:easy i}
There are $a,b\in K_\p^\times$ such that $(v_\p(a),v_\p(a+b),v_\p(a-b))=(0,1,0)$.

\item \label{L:easy ii}
There are $a,b\in K_\p^\times$ such that 
$(v_\p(a),v_\p(a+b),v_\p(a-b))=(-1,1,-1)$, $-2a(a+b)$ is a square in $K_\p$, and $(a-b)\omega$ is a square in $K_\p$.

\item \label{L:easy iii}
There are $a,b\in K_\p^\times$ such that 
$(v_\p(a),v_\p(a+b),v_\p(a-b))=(-1,1,-1)$, $-2a(a+b)$ is a square in $K_\p$, and $(a-b)\omega$ is not a square in $K_\p$.

\item \label{L:easy iv}
There are $a,b\in K_\p^\times$ such that  $(v_\p(a),v_\p(a+b),v_\p(a-b))=(0,2,0)$ and $-2a(a+b)$ is a square in $K_\p$.
\item \label{L:easy v}
There are $a,b\in K_\p^\times$ such that  $(v_\p(a),v_\p(a+b),v_\p(a-b))=(0,0,1)$ and $a$ is not a square in $K_\p$.
\end{romanenum}
\end{lemma}
\begin{proof}
For (\ref{L:easy i}), we can take $(a,b)=(1,-1-\omega)$.  For (\ref{L:easy ii}) and (\ref{L:easy iii}), we can take $(a,b)=(\tfrac{1}{2}u\omega^{-1},-\tfrac{1}{2}u\omega^{-1}-u^{-1} \omega)$ with $u\in \OO_\p^\times$; $(a-b)\omega$ is a square in $K_\p$ if and only if $u$ is a square modulo $\p$.  For (\ref{L:easy iv}), we can take $(a,b)=(1,-1-\tfrac{1}{2} \omega^2)$.  For (\ref{L:easy v}), we can take $(a,b)=(u,u+\omega)$ where $u\in \OO_\p^\times$ is not a square modulo $\p$.
\end{proof}

\section{Proof of Theorem~\ref{T:true}} \label{S:main proof}

Define $L:=K(\sqrt{D})$; it is a quadratic extension of $K$.   Let $S_0$ be a finite set of nonzero prime ideals of $\OO_K$ that contains all those that divide $2$, are ramified in $L$ or satisfy $v_\p(D)\neq 0$.  By increasing the finite set $S_0$ appropriately, we may assume that $\OO_{K,S_0}$ is a PID.  For a nonzero prime ideal $\p\notin S_0$ of $\OO_K$, $D$ is a square in $K_\p$ if and only if $\p$ splits in $L$.  

Let $\n$ be the conductor of the abelian extension $L/K$ and let $\n_0\subseteq \OO_K$ be the corresponding integral ideal.  The prime ideals of $\OO_K$ that divide $\n_0$ are precisely those that are ramified in $L$.  Let $\Cl_K^{\n}$ be the ray class group for the modulus $\n$.
By class field theory, there is a unique  group homomorphism 
\[
\psi\colon \Cl_K^{\n}\to \{\pm 1\}
\]
such that for any nonzero prime ideal $\p\nmid \n_0$ of $\OO_K$ we have $\psi([\p])=1$ if and only if $\p$ splits in $L$.

Let $S_\infty$ be the set of infinite places of $K$.      

\subsection{The curves} \label{SS:The curves}
Consider fixed $a,b \in K$ for which $ab(a+b)(a-b)\neq 0$ and hence $a$, $a+b$ and $a-b$ are distinct and nonzero.   Recall that for $d\in K^\times$, we defined an elliptic curve $E_d$ over $K$ by
\begin{align} \label{E:main proof}
y^2=x^3+4a(a+b) d\,  x^2 +2a(a+b)^2(a-b)  d^2\, x.
\end{align}
From \S\ref{SS:2-descent}, there is a degree $2$ isogeny $\phi_d \colon E_d\to E'_d$ with $\phi_d((0,0))=0$, where $E'_d$ is the elliptic curve over $K$ defined by the Weierstrass equation
\begin{align}\label{E:main2 proof}
y^2=x^3-8a(a+b) d \, x^2 +8a(a+b)^3  d^2 \, x.
\end{align}
For each place $v$ of $K$, we have a homomorphism
\[
\delta_{d,v}\colon E'_d(K_v) \to K_v^\times/(K_v^\times)^2 
\]
as in in \S\ref{SS:our curves 1} with kernel $\phi_d(E_d(K_v))$.   For our proof, we will be interested in the case where $d$ is either $1$ or $D$.  Note that $a$ and $b$ are implicit in the notation $E_d$ and $\delta_{d,v}$.

\subsection{The set of primes $S$}

In this section, we construct a finite set $S$ of nonzero prime ideals of $\OO_K$ and a nonempty open subset $U_\p$ of $K_\p^2$ for each $\p\in S$.   For each $\p\in S$, we will also construct subgroups $\Phi_{1,\p}$ and $\Phi_{D,\p}$ of $K_\p^\times/(K_\p^\times)^2$.    The complicated choices in this section are motivated by the explicit $2$-descent computations that will be done later on (see especially Lemmas~\ref{L:Selmer ratio calculation} and \ref{L:first Selmer}).

For each infinite place $v\in S_\infty$, we define $\Phi_{1,v}$ and $\Phi_{D,v}$ to be the trivial subgroup of $K_v^\times/(K_v^\times)^2$.

\begin{lemma} \label{L:set low conditions}
There is a finite set $S_1$ of nonzero prime ideals of $\OO_K$ disjoint from $S_0$ such that for each  $\p\in S_0\cup S_1$ we have a nonempty open subset $U_\p$ of $K_\p^2$ and subgroups $\Phi_{1,\p}$ and $\Phi_{D,\p}$ of $K_\p^\times/(K_\p^\times)^2$ for which the following hold:
\begin{alphenum}
\item \label{L:set low conditions a}
If $(a,b)\in K^2$ lies in $U_\p$ for a prime ideal $\p\in S_0 \cup S_1$, then $ab(a+b)(a-b)\neq 0$ and we have
\[
\operatorname{Im}(\delta_{1,\p})=\Phi_{1,\p} \quad \text{ and } \quad \operatorname{Im}(\delta_{D,\p})=\Phi_{D,\p}.
\] 
Note that $\delta_{1,\p}$ and $\delta_{D,\p}$ depend on $(a,b)$ while the groups $\Phi_{1,\p}$ and $\Phi_{D,\p}$ do not.  

\item\label{L:set low conditions b}
For all $\p\in S_0\cup S_1$ and $(a,b)\in U_\p$, the integers $v_\p(a)$, $v_\p(a+b)$ and $v_\p(a-b)$ do not depend on the choice of $(a,b)$.   

\item\label{L:set low conditions c}
For $\p\in S_0$ and $(a,b)\in U_\p$, we have $v_\p(a)=v_\p(a+b)=v_\p(a-b)=0$.

\item \label{L:set low conditions d}
If $(a,b)\in K^2$ lies in $U_\p$ for all $\p\in S_0$, then $a+b\in 1+\n_0\OO_\p$ for all prime ideals $\p|\n_0$. 

\item \label{L:set low conditions e}
If $(a,b)\in K^2$ lies in $U_\p$ for all $\p\in S_1$, then
\[
\sum_{\p\in S_1,\, \p \text{ inert in $L$}} v_\p(a+b) \equiv 1 \pmod{2}.
\]

\item  \label{L:set low conditions f}
We have $\prod_{\p\in S_0 \cup S_1}  |\Phi_{1,\p}| = \prod_{\p\in S_0  \cup S_1}  |\Phi_{D,\p}|.$

\item \label{L:set low conditions g}
For $d\in \{1,D\}$, the images of the groups
\[
\OO_{K,S_0\cup S_1}^\times,\quad  \prod_{v\in S_0 \cup S_1 \cup S_\infty} \Phi_{d,v} \quad \text{ and }\quad \prod_{v\in S_0\cup S_\infty} \{1\} \times \prod_{v\in S_1} \OO_v^\times
\]
  in $\prod_{v\in S_0 \cup S_1 \cup S_\infty} K_v^\times/(K_v^\times)^2$ generate $\prod_{v\in S_0 \cup S_1 \cup S_\infty} K_v^\times/(K_v^\times)^2$.
\end{alphenum}
\end{lemma}
\begin{proof}
Consider a nonzero prime ideal $\p$ of $\OO_K$ and a choice of $a_\p,b_\p \in K_v$ with $a_\p b_\p(a_\p+b_\p)(a_\p-b_\p)\neq 0$.   For a fixed real number $\epsilon_\p>0$, we define the open subset 
\[
U_\p:=\{(a,b)\in K_\p^2:  |a-a_\p|_\p<\epsilon_\p,\, |b-b_\p|_\p<\epsilon_\p\}
\] 
of $K_\p^2$.   By taking $\epsilon_\p>0$ sufficiently small, we may assume that $ab(a-b)(a+b)\neq 0$ for all $(a,b)\in U_\p$.  We may also take $\epsilon_\p>0$ sufficiently small so that the $\p$-adic valuations of $a$, $a+b$ and $a-b$ do not depend on the choice of $(a,b)\in U_\p$.  For each $d\in \{1,D\}$ and $(a,b) \in U_\p$, we obtain an elliptic curve $E_{d,(a,b)}'$ over $K_\p$ defined by the equation (\ref{E:main2 proof}).  Following the cohomological construction of \S\ref{SS:2-descent}, we obtain a homomorphism
\begin{align} \label{E:Kv connecting map}
E_{d,(a,b)}'(K_\p) \to K_\p^\times/(K_\p^\times)^2
\end{align}
that takes $(0,0)$ to $2a(a+b)\cdot (K_\p^\times)^2$ and takes any point $(x,y) \notin \{0,(0,0)\}$ to $x\cdot (K_\p^\times)^2$.  The key topological observation is the image of (\ref{E:Kv connecting map}) does not depend on the choice $(a,b)$ if we take $\epsilon_\p>0$ sufficiently small.   Thus by taking $\epsilon_\p>0$ sufficiently small, we may assume that (\ref{E:Kv connecting map}) has a common image $\Phi_{d,\p}$ for all $(a,b)\in U_\p$.   For the rest of the proof if we specify a prime ideal $\p$ and a pair $(a_\p,b_\p)$ with $a_\p b_\p(a_\p+b_\p)(a_\p-b_\p)\neq 0$, then we will have $U_\p$, $\Phi_{1,\p}$ and $\Phi_{D,\p}$ with properties as above.   Parts (\ref{L:set low conditions a}) and (\ref{L:set low conditions b}) will be immediate once we have defined the set $S_1$ and chosen $a_\p,b_\p\in K_\p$ for all $\p\in S_0\cup S_1$.

For $\p\in S_0$, we take $a_\p=1$ and choose $b_\p$ to be a nonzero element of $\n_0\p\OO_\p$; in particular, $v_\p(a_\p)=v_\p(a_\p+b_\p)=v_\p(a_\p-b_\p)=0$.   Take any $(a,b)\in K^2$ that lies in $U_\p$ for all $\p\in S_0$.  For each $\p|\n_0$, we have $\p\in S_0$ and hence $a+b \in 1+ \n_0 \OO_\p$.  We thus have (\ref{L:set low conditions c}) and (\ref{L:set low conditions d}).
  
Choose a finite set $S_1$ of nonzero prime ideals $\p\notin S_0$ of $\OO_K$ that are inert in $L=K(\sqrt{D})$ along with values $\xi_\p\in \{1/2,2\}$ such that 
\begin{align} \label{E:xi defn}
\prod_{\p \in S_0} |\Phi_{1,\p}|/|\Phi_{D,\p}| \cdot \prod_{\p \in S_1} \xi_\p =1
\end{align}
holds and $\xi_\p=2$ for some $\p \in S_1$.  Note that $D$ is not a square in $K_\p$ for $\p\in S_1$ since these primes are inert in $L$.  For $\p\in S_1$, we will now choose a pair $(a_\p,b_\p)\in K_\p^2$ using Lemma~\ref{L:local existence}. For each $\p \in S_1$ with $\xi_\p=1/2$, we define $\mu_\p:=1$ and choose $a_\p,b_\p\in K_\p^2$ with $a_\p b_\p(a_\p+b_\p)(a_\p-b_\p)\neq 0$ and
\[
(v_\p(a_\p),v_\p(a_\p+b_\p),v_\p(a_\p-b_\p))=(0,1,0).
\]    
For each $\p\in S_1$ with $\xi_\p=2$, we choose a value $\mu_\p \in \{1,2\}$.   For $\p\in S_1$ with $\xi_\p=2$ and $\mu_\p=1$, we choose $a_\p,b_\p\in K_\p$ with $a_\p b_\p(a_\p+b_\p)(a_\p-b_\p)\neq 0$,
\[
(v_\p(a_\p),v_\p(a_\p+b_\p),v_\p(a_\p-b_\p))=(-1,1,-1)
\] 
and $-2a_\p(a_\p+b_\p)$ a square in $K_\p$.  For $\p\in S_1$ with $\xi_\p=2$ and $\mu_\p=2$, we choose $a_\p,b_\p\in K_\p$ with $a_\p b_\p(a_\p+b_\p)(a_\p-b_\p)\neq 0$, 
\[
(v_\p(a_\p),v_\p(a_\p+b_\p),v_\p(a_\p-b_\p))=(0,2,0)
\]
and $-2a_\p(a_\p+b_\p)$ a square in $K_\p$.   For each $\p\in S_1$ and any $(a,b)\in U_\p \cap K^2$, we can use Lemma~\ref{L:Selmer}(\ref{L:Selmer iv}--\ref{L:Selmer ix}) with the above $\p$-adic valuations and $D\notin (K_\p^\times)^2$ to compute $|\operatorname{Im}(\delta_{1,\p})|/|\operatorname{Im}(\delta_{D,\p})|$ and show that $|\operatorname{Im}(\delta_{1,\p})|/|\operatorname{Im}(\delta_{D,\p})| = \xi_\p$.   Therefore, $|\Phi_{1,\p}|/|\Phi_{D,\p}|=\xi_\p$ for all $\p\in S_1$.  So equation (\ref{E:xi defn}) becomes
\[
\prod_{\p \in S_0\cup S_1} |\Phi_{1,\p}|= \prod_{\p \in S_0\cup S_1} |\Phi_{D,\p}|.
\]
If $(a,b)\in K^2$ lies in $U_\p$ for all $\p\in S_1$, then
\[
\sum_{\p\in S_1,\, \p \text{ inert in $L$}} v_\p(a+b)= \sum_{\p\in S_1} v_\p(a_\p+b_\p) = \sum_{\p\in S_1} \mu_\p.
\]
Since we have the freedom to choose $\mu_\p\in \{1,2\}$ for $\p\in S_1$ with $\xi_\p=2$, we may assume that the $\mu_\p$ have been chosen so that
\begin{align}
\sum_{\p\in S_1,\, \p \text{ inert in $L$}} v_\p(a+b) \equiv 1 \pmod{2}
\end{align}
holds for all $(a,b)\in K^2$ that lie in $U_\p$ for all $\p\in S_1$.    With our set $S_1$, we have now proved (\ref{L:set low conditions e}) and (\ref{L:set low conditions f}).

By weak approximation for $K$, there is a finite set $S_1'$ of nonzero prime ideals of $\OO_K$ that is disjoint with $S_0\cup S_1$ such that the homomorphism
\begin{align} \label{E:S1' map}
\OO_{K,S_0\cup S_1\cup S_1'}^\times \to \prod_{v\in S_0 \cup S_1 \cup S_\infty} K_v^\times/(K_v^\times)^2
\end{align}
is surjective. For each $\p \in S_1'$, we choose $a_\p,b_\p \in K_\p$ with $a_\p b_\p (a_\p+b_\p)(a_\p-b_v\p)\neq 0$ such that $(v_\p(a_\p),v_\p(a_\p+b_\p),v_\p(a_\p-b_\p))=(1,0,0)$.    We then have $U_\p$, $\Phi_{1,\p}$ and $\Phi_{D,\p}$ as above.   For each $\p\in S_1'$, Lemma~\ref{L:Selmer}(\ref{L:Selmer iii}) implies that $\Phi_{1,\p}$ and $\Phi_{D,\p}$ are subgroups of $K_\p^\times/(K_\p^\times)^2$ of order $2$ that are not equal to $\OO_\p^\times/(\OO_\p^\times)^2$.   So for $d\in \{1,D\}$, the subgroup of $\prod_{v\in S_0\cup S_1\cup S_1' \cup S_\infty} K_v^\times$ generated by the groups
\begin{align} \label{E:S1' gen}
\prod_{v\in S_0\cup S_1 \cup S_1' \cup S_\infty} \Phi_{d,v} \quad \text{ and }\quad \prod_{v\in S_0\cup S_1\cup S_\infty} \{1\} \times \prod_{v\in S_1'} \OO_v^\times
\end{align}
 contains $\prod_{v\in S_0\cup S_1\cup S_\infty} \{1\} \times \prod_{v\in S_1'} K_v^\times$.    Since the homomorphism (\ref{E:S1' map}) is surjective, we deduce that $\prod_{v\in S_0\cup S_1\cup S_1' \cup S_\infty} K_v^\times$ is generated by the image of the groups (\ref{E:S1' gen}) and the image of $\OO_{K,S_0\cup S_1\cup S_1'}^\times$.  Therefore, (\ref{L:set low conditions g}) holds when we replace $S_1$ by $S_1\cup S_1'$.   Observe that (\ref{L:set low conditions e}) and (\ref{L:set low conditions f}) remain true after replacing $S_1$ by $S_1\cup S_1'$ since we have $v_\p(a_\p+b_\p)=0$ and $|\Phi_{1,\p}|=2=|\Phi_{D,\p}|$ for all $\p\in S_1'$.
\end{proof}
 
For the rest of the proof, we fix $S_1$, the sets $U_\p$, and the groups $\Phi_{1,\p}$ and $\Phi_{D,\p}$ as in Lemma~\ref{L:set low conditions}.

\begin{lemma} \label{L:last generators}
There is a finite set $S_2$ of nonzero prime ideals of $\OO_K$ that is disjoint from $S_0\cup S_1$ such that the following hold:
\begin{itemize}
\item
all the prime ideals in $S_2$ split in $L$,
\item
the natural homomorphism
\begin{align} \label{E:early gammad}
\OO_{K,S_0\cup S_1\cup S_2}^\times \to \prod_{v\in S_0\cup S_1\cup S_\infty} \big(K_{v}^\times/(K_{v}^\times)^2\big)/\Phi_{d,v}
\end{align}
is surjective for each $d\in \{1,D\}$,
\item
$S_2$ contains a prime ideal $\pi_1 \OO_K$ that splits in $L$ such that $\pi_1\in \OO_K$ is a square in $K_v$ for all places $v\in S_0\cup S_1\cup S_\infty$.  
\end{itemize}
\end{lemma}
\begin{proof}
Let $\m_0$ be a nonzero ideal of $\OO_K$ whose prime divisors are the prime ideals $\p\in S_0 \cup S_1$.   We may assume that $\m_0$ is divisible by $\n_0$.   Let $\m$ be the modulus obtained from $\m_0$ and all the real places of $K$.

For each fractional ideal $I$ of $K$ we have $I=\prod_{\p} \p^{v_\p(I)}$, where the product is over the nonzero prime prime ideals of $\OO_K$ and the $v_\p(I)$ are unique integers that are zero for all but finitely many $\p$.  Let $\calI_K^{\m}$ be the group of fractional ideals $I$ of $K$ with $v_\p(I)=0$ for all $\p|\m_0$.  Define $K^{\m}$ to be the group of $\alpha\in K^\times$ with $v_\p(\alpha)=0$ for all $\p|\m_0$.  Let $K^{\m,1}$ be the subgroup of $K^{\m}$ consisting of $\alpha\in K^{\m}$ with $v_\p(\alpha-1) \geq v_\p(\m_0)$ for all $\p| \m_0$ and with $\alpha$ positive in $K_v$ for all real places $v$ of $K$.  Let $\calP_{K}^{\m}$ be the subgroup of $\calI_K^{\m}$ consisting of $\alpha \OO_K$ with $\alpha\in K^{\m,1}$.   Recall that the \emph{ray class group} of $\m$ is the (finite) group $\Cl_K^{\m}:=\calI_K^{\m}/\calP_{K}^{\m}$.  We have a natural exact sequence
\begin{align} \label{E:ray class group exact sequence}
\OO_K^\times \to K^{\m}/K^{\m,1} \xrightarrow{f} \Cl_K^{\m} \to \Cl_K \to 1,
\end{align}
where $\Cl_K$ is the class group of $K$.   Assuming that the powers of the prime divisors of $\m_0$ were chosen large enough, we may assume that we have a well-defined homomorphism
\begin{align} \label{E:Km map}
K^{\m}/K^{\m,1} \to \prod_{\p\in S_0 \cup S_1} \OO_\p^\times/(\OO_\p^\times)^2 \times \prod_{v\in  S_\infty} K_v^\times/(K_v^\times)^2;
\end{align}
this map is surjective by weak approximation for $K$.

Let $F$ be the corresponding \emph{ray class field} of $\m$; it is a finite abelian extension of $K$ with an isomorphism $\Gal(F/K)\xrightarrow{\sim} \Cl_K^{\m}$ so that $F$ is unramified at each nonzero prime ideal $\p\nmid \m_0$ of $\OO_K$ and $\Frob_\p \in \Gal(F/K)$ is sent to the class $[\p]\in \Cl_K^{\m}$.

Let $\psi'\colon \Cl_K^{\m}\to \{\pm 1\}$ be the homomorphism obtained by composing the quotient map $\Cl_K^{\m}\to \Cl_K^{\n}$ with $\psi$.   We have $L\subseteq F$ and for a nonzero prime ideal $\p\notin S_0\cup S_1$, $\p$ splits in $L$ if and only if $\psi'([\p])=1$.    Let $G\subseteq \Cl_K^{\m}$ be the intersection of $\ker (\psi')$ with $f(K^{\m}/K^{\m,1})$.  We have $[f(K^{\m}/K^{\m,1}):G]\leq 2$.    By the Chebotarev density theorem, we can find a finite set $S_2$ of nonzero prime ideals of $\OO_K$ that is disjoint from $S_0\cup S_1$ such that the set $\{[\p]: \p \in S_2\} \subseteq \Cl_{K}^{\m}$ generates the group $G$.  Note that all the prime ideals $\p \in S_2$ are principal and split in $L$.   

By our choice of $S_2$, the exact sequence (\ref{E:ray class group exact sequence}) and the surjectivity of (\ref{E:Km map}), we find that the quotient homomorphism 
\begin{align} \label{E:OKS1 unit map}
\OO_{K,S_2}^\times \to K^{\m}/K^{\m,1} \to \prod_{\p\in S_0 \cup S_1} \OO_\p^\times/(\OO_\p^\times)^2 \times \prod_{v\in  S_\infty} K_v^\times/(K_v^\times)^2
\end{align}
has cokernel of cardinality at most $2$.  If (\ref{E:OKS1 unit map}) is not surjective, then since $\n_0$ has only prime divisors in $S_0$ we find that composing (\ref{E:OKS1 unit map}) with the quotient map to $\prod_{\p\in S_0} \OO_\p^\times/(\OO_\p^\times)^2 \times \prod_{v\in  S_\infty} K_v^\times/(K_v^\times)^2$ gives another homomorphism that is not surjective.  Therefore, the image of (\ref{E:OKS1 unit map}) must contain $\prod_{v\in S_0 \cup S_\infty} \{1\} \times \prod_{\p\in S_1} \OO_\p^\times/(\OO_\p^\times)^2$.    By Lemma~\ref{L:set low conditions}(\ref{L:set low conditions g}), we deduce that (\ref{E:early gammad}) is surjective for each $d\in \{1,D\}$.

By the Chebotarev density theorem, there is a nonzero prime ideal $\p\notin S_0\cup S_1\cup S_2$ of $\OO_K$ that splits completely in $F$ and hence $[\p]=1$ in $\Cl_K^{\m}$.  The prime $\p$ splits in $L$ since $L\subseteq F$.  We have $\p=\pi_1 \OO_K$ for some $\pi_1 \in \OO_K$.  By the exact sequence (\ref{E:ray class group exact sequence}), there is a unit $u \in \OO_K^\times$ for which $u \cdot (K^{\m,1})=\pi_1\cdot (K^{\m,1})$.  After replacing $\pi_1$ by $u^{-1}\pi_1$, we may assume that $\pi_1 \in K^{\m,1}$ and hence $\pi_1$ is a square in $K_v$ for all $v\in S_0\cup S_1 \cup S_\infty$ by our choice of $\m$.   The lemma follows by adding $\p=\pi_1\OO_K$ to $S_2$.
\end{proof}

From now on, we fix a set $S_2$ and a prime $\pi_1\in \OO_K$ as in Lemma~\ref{L:last generators}.  Take any $d\in \{1,D\}$.   We have a natural group homomorphism
\[
\gamma_d \colon \OO_{K,S_0\cup S_1\cup S_2}^\times/(\OO_{K,S_0\cup S_1\cup S_2}^\times)^2 \to \prod_{v\in S_0\cup S_1\cup S_\infty} \big(K_{v}^\times/(K_{v}^\times)^2\big)/\Phi_{d,v}
\]
that is surjective by our choice of $S_2$.  Define the group $V_d :=\ker \gamma_d$.

\begin{lemma} \label{L:prelim Selmer product}
\begin{romanenum}
\item \label{L:prelim Selmer product i}
For each $d\in \{1,D\}$, we have $|V_d| = 2^{|S_2|}\cdot \prod_{v\in S_0 \cup S_1\cup S_\infty} \tfrac{1}{2} |\Phi_{d,v}|$.  
\item \label{L:prelim Selmer product ii}
We have $|V_1|=|V_D|$.
\end{romanenum}
\end{lemma}
\begin{proof}
Let $r$ be the number of real embeddings of $K$ and let $s$ be the number of pairs of conjugate complex embeddings of $K$.  From the definition of $V_d$ and the surjectivity of $\gamma_d$, we have
\begin{align} \label{E:2t}
|V_d|=|\OO_{K,S_0\cup S_1 \cup S_2}^\times/(\OO_{K,S_0 \cup S_1 \cup S_2}^\times)^2|\cdot \prod_{v\in S_0\cup S_1\cup S_\infty}\big|\Phi_{d,v}\big|  \cdot \prod_{v\in S_0\cup S_1\cup S_\infty} \big|K_v^\times/(K_v^\times)^2\big|^{-1}.
\end{align}

Consider any place $v|2$ of $K$.   Let $e_v$ and $f_v$ be the ramification index and inertia degree, respectively, of the extension $K_v/\QQ_2$.  We have $[K_v:\QQ_2]=e_vf_v$.  We have $\OO_v^\times\cong C_v \times \ZZ_2^{e_vf_v}$, where $C_v$ is a finite cyclic group, cf.~\cite[II 5.7]{Neukirch}.  The group $C_v$ has even cardinality since it contains $-1$ and hence $\OO_v^\times/(\OO_v^\times)^2\cong (\ZZ/2\ZZ)^{1+e_vf_v}$.   Therefore, $|K_v^\times/(K_v^\times)^2| = 2^{2+e_v f_v}$.  Since $\sum_{v|2} e_v f_v = [K:\QQ]=r+2s$, we deduce that $\prod_{v|2} |K_v^\times/(K_v^\times)^2| = 2^{r+2s} \prod_{v|2} 4$.   We also have $\prod_{v\in S_\infty} |K_v^\times/(K_v^\times)^2|=2^r$.  For any place $v\in S_0\cup S_1$ with $v\nmid 2$, we have $|K_v^\times/(K_v^\times)^2\big|=4$.  Therefore, 
\[
\prod_{v\in S_0\cup S_1\cup S_\infty} | K_v^\times/(K_v^\times)^2| = 2^{2r+2s} 4^{|S_0|} 4^{|S_1|}.
\]

Let $t$ be the rank of the finite generated abelian group $\OO_{K,S_0\cup S_1\cup S_2}^\times$.   By Dirichlet's unit theorem, we have $t=r+s-1+|S_0|+|S_1|+|S_2|$. Since the torsion subgroup of $\OO_{K,S_0\cup S_1\cup S_2}^\times$ is a finite cyclic group of even cardinality, we find that $\OO_{K,S_0\cup S_1\cup S_2}^\times/(\OO_{K,S_0\cup S_1\cup S_2}^\times)^2$ has dimension $t+1$ over $\FF_2$.    Therefore, 
\[
|\OO_{K,S_0\cup S_1\cup S_2}^\times/(\OO_{K,S_0\cup S_1\cup S_2}^\times)^2| = 2^{t+1}=2^{r+s+|S_0|+|S_1|+|S_2|}. 
\]
From (\ref{E:2t}) and the above computations, we have
\[
|V_d|= 2^{r+s+|S_0|+|S_1|+|S_2|}  \cdot \prod_{v \in S_0\cup S_1\cup S_\infty}|\Phi_{d,v}| \cdot 2^{-2r-2s} 2^{-2|S_0|}2^{-2|S_1|}=2^{|S_2|} \prod_{v\in S_0\cup S_1\cup S_\infty}\tfrac{1}{2}|\Phi_{d,v}|.
\]
This completes the proof of (\ref{L:prelim Selmer product i}).  Since $\Phi_{1,v}=\Phi_{D,v}=1$ for all $v\in S_\infty$, part (\ref{L:prelim Selmer product ii}) follows from (\ref{L:prelim Selmer product i}) and Lemma~\ref{L:set low conditions}(\ref{L:set low conditions f}).
\end{proof}

 Note that $\pi_1\cdot (\OO_{K,S_0\cup S_1\cup S_2}^\times)^2$ lies in $V_1$ and $V_D$ since $\pi_1\OO_K \in S_2$ and $\pi_1$ is a  square in $K_v$ for all $v\in S_0\cup S_1\cup S_\infty$.  Thus there is an integer $n\geq 0$ such that $|V_1|=2^{n+1}$.  By Lemma~\ref{L:prelim Selmer product}(\ref{L:prelim Selmer product ii}), we have $|V_D|=2^{n+1}$.  
 
 We can view $V_1$ and $V_2$ as subgroups of $K^\times/(K^\times)^2$.  For each nonzero prime ideal $\p$ of $\OO_K$, let 
\[
\varphi_{\p}\colon K^\times/(K^\times)^2\to K_{\p}^\times/(K_{\p}^\times)^2
\] 
be the natural quotient map.  

\begin{lemma} \label{L:S3}
There are distinct nonzero prime ideals $\p_1,\ldots, \p_n\notin S_0\cup S_1 \cup S_2$ of $\OO_K$ such that 
$V_1 \cap \bigcap_{i=1}^{n}\ker \varphi_{\p_i}$ and $V_D \cap \bigcap_{i=1}^{n}\ker \varphi_{\p_i}$ are both equal to the group generated by $\pi_1\cdot  (K^\times)^2$.
\end{lemma}
\begin{proof}
We shall view $V_1$ and $V_D$ as vector spaces over $\FF_2$.  Take $m\geq 0$ maximal for which there are distinct nonzero prime ideals $\p_{1},\ldots, \p_{m} \notin S_0\cup S_1\cup S_2$ of $\OO_K$ such that 
\begin{itemize}
\item
$W_1:=V_1\cap \bigcap_{i=1}^{m} \ker \varphi_{\p_i}$ has codimension $m$ in $V_1$ and contains $\pi_1 \cdot(K^\times)^2$,
\item
$W_D:=V_D\cap \bigcap_{i=1}^{m} \ker \varphi_{\p_i}$ has codimension $m$ in $V_D$ and contains $\pi_1\cdot (K^\times)^2$,
\end{itemize}
This is well-defined since the case $m=0$ is always possible; $V_1$ and $V_D$ contain the square class of $\pi_1$.   Since $\dim_{\FF_2} V_1 =\dim_{\FF_2} V_D= n+1$ and $\pi_1\cdot(K^\times)^2\neq 1$, we must have $m\leq n$.

Suppose that $m<n$ and hence $W_1$ and $W_D$ have dimension $n+1-m \geq 2$ over $\FF_2$.    There are units $u,v\in \OO_{K,S_0\cup S_1\cup S_2}^\times$ such that $u\cdot(K^\times)^2$ and $v\cdot(K^\times)^2$ lie in $W_1$ and $W_D$, respectively, and both do not lie in the group generated by $\pi_1\cdot(K^\times)^2$.  The subgroup of $K^\times/(K^\times)^2$ generated by $u$, $v$ and $\pi_1$ either has dimension $3$ over $\FF_2$ or $uv \cdot(K^\times)^2$ lies in the group $\langle \pi_1 \cdot(K^\times)^2\rangle$.   By the Chebotarev density theorem applied to the extension $K(\sqrt{u},\sqrt{v},\sqrt{\pi_1})/K$, we find that there is a nonzero prime ideal $\p \notin S_0\cup S_1\cup S_2\cup \{\p_1,\ldots, \p_m\}$ such that $u$ and $v$ are not squares in $K_\p$ and $\pi_1$ is a square in $K_\p$.   However with $\p_{m+1}:=\p$, this contradicts the maximality of $m$.   Therefore, $m=n$ and the lemma follows.
\end{proof}

For the rest of the proof, we fix prime ideals $\p_1,\ldots, \p_n$ as in Lemma~\ref{L:S3}.   

\begin{lemma} \label{L:p0 existence}
There is a nonzero prime ideal $\p_0\notin S_0\cup S_1 \cup S_2 \cup \{\p_1,\ldots, \p_n\}$ of $\OO_K$ such that every element of $\OO_{K,S_0\cup S_1 \cup S_2}^\times$ is a square in $K_{\p_0}$.
\end{lemma}
\begin{proof}
Let $u_1,\ldots, u_m$ be a set of generators of the group $\OO_{K,S_0\cup S_1 \cup S_2}^\times$; it is finitely generated by Dirichlet's unit theorem.   Define $F:=K(\sqrt{u_1},\ldots, \sqrt{u_m})$.   Then any nonzero prime ideal $\p\notin S_0\cup S_1\cup S_2\cup \{\p_1,\ldots, \p_n\}$ of $\OO_K$ that splits completely in $F$ will have the desired properties; the existence of such primes follows from the Chebotarev density theorem.
\end{proof}

For the rest of the proof, we fix a prime $\p_0$ as in Lemma~\ref{L:p0 existence}.   Define $S_3:=\{\p_0,\p_1,\ldots, \p_n\}$.   Define the finite set
\[
S:=S_0\cup S_1\cup S_2 \cup S_3
\]
which consists of nonzero prime ideals of $\OO_K$.  Fix a value $\varepsilon \in \{\pm 1\}$; we will make a specific choice later.\\

Take any prime ideal $\p\in S_2$.     Let $U_\p$ be the set of pairs $(a,b) \in K_\p^2$ that satisfy the following conditions:
\begin{itemize}
\item 
$ab(a+b)(a-b)\neq 0$, 
\item
$(v_\p(a),v_\p(a+b),v_\p(a-b))=(-1,1,-1)$,
\item
$-2a(a+b)$ is a square in $K_\p$,
\item
if $\p=\pi_1\OO_K$, then $(a-b)\pi_1$ is a square in $K_\p$ when $\varepsilon=1$ and $(a-b)\pi_1$ is not a square in $K_\p$ when $\varepsilon=-1$.
\end{itemize}
The set $U_\p$ is open in $K_\p^2$.  The set $U_\p$ is nonempty by Lemma~\ref{L:local existence} (when $\p=\pi_1\OO_K$, take the uniformizer $\omega:=\pi_1$ of $K_\p$). We define $\Phi_{1,\p}=K_\p^\times/(K_\p^\times)^2$ and $\Phi_{D,\p}=K_\p^\times/(K_\p^\times)^2$.

Take any prime ideal $\p\in S_3$.     Let $U_\p$ be the set of pairs $(a,b) \in K_\p^2$ that satisfy the following conditions:
\begin{itemize}
\item 
$ab(a+b)(a-b)\neq 0$, 
\item
$(v_\p(a),v_\p(a+b),v_\p(a-b))=(0,0,1)$,
\item
if $\p=\p_0$, then $a$ is not a square in $K_{\p}$.
\end{itemize}
The set $U_\p$ is nonempty by Lemma~\ref{L:local existence}.  We define $\Phi_{1,\p}$ and $\Phi_{D,\p}$ to be the trivial subgroup of $K_\p^\times/(K_\p^\times)^2$.  We now show that these new sets $U_\p$ have the same property as before.

\begin{lemma} \label{L:Phidp validation}
Take any $\p\in S$ and any $(a,b) \in K^2$ that lies in $U_\p$.  Then $ab(a+b)(a-b)\neq 0$ and 
\[
\operatorname{Im}(\delta_{1,\p})=\Phi_{1,\p} \quad \text{ and } \quad \operatorname{Im}(\delta_{D,\p})=\Phi_{D,\p}.
\] 
Note that $\delta_{1,\p}$ and $\delta_{D,\p}$ depend on $(a,b)$ while the groups $\Phi_{1,\p}$ and $\Phi_{D,\p}$ do not.  
\end{lemma}
\begin{proof}
When $\p\in S_0\cup S_1$, this is Lemma~\ref{L:set low conditions}(\ref{L:set low conditions a}).

Suppose that $\p \in S_2$ and take any $d\in \{1,D\}$.   By our choice of $U_\p$, we have $(v_\p(a),v_\p(a+b),v_\p(a-b))=(-1,1,-1)$ and $-2a(a+b)$ is a square in $K_\p$.  Also $-2a(a+b)D$ is a square in $K_\p$ since every prime in $S_2$ splits in $L$.   Therefore, $|\operatorname{Im}(\delta_{d,\p})| =4$ by Lemma~\ref{L:Selmer}(\ref{L:Selmer vi}).   Thus $\operatorname{Im}(\delta_{d,\p})=K_\p^\times/(K_\p^\times)^2=\Phi_{d,\p}$.

Suppose that $\p \in S_3$ and take any $d\in \{1,D\}$.  By our choice of $U_\p$, we have $(v_\p(a),v_\p(a+b),v_\p(a-b))=(0,0,1)$.   Therefore, $|\operatorname{Im}(\delta_{d,\p})| =1$ by Lemma~\ref{L:Selmer}(\ref{L:Selmer ii}).   Thus $\operatorname{Im}(\delta_{d,\p})=1=\Phi_{d,\p}$.
\end{proof}

\subsection{A choice of elliptic curve and a $2$-descent}

With a fixed $\varepsilon\in \{\pm 1\}$ still to be determined, we have constructed a finite set $S$ of prime ideals and we have a nonempty open subset $U_\p$ of $K_\p^2$ for each $\p \in S$.   We also have a specific prime $\pi_1 \in \OO_K$ that generates a prime ideal in $S$.

For the rest of the proof we take any $a,b\in \OO_K$ for which the following hold:
\begin{itemize}
\item
$a$, $a+b$ and $a-b$ generate distinct nonzero prime ideals of $\OO_{K,S}$,
\item
$(a,b)$ lies in $U_\p$ for all $\p\in S$,
\item
$0<a<b$ in $K_v$ for all real places of $K$.
\end{itemize}
If no such $a$ and $b$ exist, then Theorem~\ref{T:true} trivially holds.  We have $ab(a+b)(a-b)\neq 0$.

For $d\in \{1,D\}$, we have an elliptic curve $E_d$ over $K$ defined by (\ref{E:main proof}) and an elliptic curve $E'_d$ over $K$ defined by (\ref{E:main2 proof}).   The curve $E_D$ is the quadratic twist of $E_1$ by $D$.   To finish the proof of Theorem~\ref{T:true} we need to show that $E_1$ has rank $1$ and $E_D$ has rank $0$.

Since $a$, $a+b$ and $a-b$ generate distinct nonzero prime ideals of $\OO_{K,S}$, 
there are distinct nonzero prime ideals $\q_1$, $\q_2$ and $\q_3$ of $\OO_K$ that are not in $S$ such that 
\begin{align} \label{E:q}
a\OO_{K,S} = \q_1 \OO_{K,S}, \quad (a+b)\OO_{K,S} = \q_2 \OO_{K,S}\quad \text{ and  }\quad (a-b)\OO_{K,S} = \q_3 \OO_{K,S}.
\end{align}
We now work out some properties for the prime ideals $\q_2$ and $\q_3$.

\begin{lemma} \label{L:pi1 not square mod q3}
For an appropriate initial choice of $\varepsilon \in \{\pm 1\}$,  $\pi_1$ is not a square in $K_{\q_3}$.
\end{lemma}
\begin{proof}
For each nonzero prime ideal $\p$ of $\OO_K$, define $e_\p:=v_\p(a-b)$.  Since $(a-b)\OO_{K,S} = \q_3 \OO_{K,S}$, we have  $(a-b)\OO_{K} = \q_3 \prod_{\p \in S} \p^{e_\p}$.  With $\p=\pi_1\OO_K$, we have $(a,b) \in U_\p$ and hence $v_\p((a-b)\pi_1)=-1+1=0$.  Therefore, 
\begin{align} \label{E:(a-b)pi}
(a-b)\pi_1\OO_{K} = \q_3 \prod_{\p \in S,\, \p\neq (\pi_1)} \p^{e_\p}= \q_3 \prod_{\p \in S_1\cup S_2\cup S_3, \,\p\neq (\pi_1)} \p^{e_\p},
\end{align}
where the last equality uses Lemma~\ref{L:set low conditions}(\ref{L:set low conditions c}).  Using second power residue symbols for the field $K$, cf.~\cite[IV \S8]{Neukirch}, we have
\[
\legendre{(a-b)\pi_1}{\pi_1}=\legendre{\pi_1}{(a-b)\pi_1} = \legendre{\pi_1}{\q_3} \cdot \prod_{\p\in S_1 \cup S_2 \cup S_3,\, \p\neq (\pi_1)} \legendre{\pi_1}{\p}^{e_\p},
\]
where the first equality uses the general reciprocity law \cite[IV Theorem 8.3]{Neukirch} and the second equality uses (\ref{E:(a-b)pi}).   Note that in applying the reciprocity law, we have used that $\pi_1$ is a square in $K_v$ for all infinite places $v$ of $K$ and $\pi_1$ is a square in $K_\p$ for all prime ideals $\p$ of $\OO_K$ dividing $2$.  With $\p=\pi_1\OO_K$, $(a,b)\in U_\p$ implies that $(a-b)\pi_1 \in \OO_\p^\times$ is a square in $K_\p$ if and only if $\varepsilon=1$.  Therefore, $\legendre{(a-b)\pi_1}{\pi_1}=\varepsilon$ and hence
\begin{align} \label{E:pi1 mod q3}
\legendre{\pi_1}{\q_3}= \varepsilon  \prod_{\p\in S_1 \cup S_2 \cup S_3,\, \p\neq (\pi_1)} \legendre{\pi_1}{\p}^{e_\p}.
\end{align}
If $\p\in S_2$, $(a,b)\in U_\p$ implies that $e_\p=-1$.  If $\p\in S_3$, $(a,b)\in U_\p$ implies that $e_\p=1$.  If $\p\in S_1$, $(a,b)\in U_\p$ implies that $e_\p$ does not depend on the choice of $a$ and $b$ by Lemma~\ref{L:set low conditions}(\ref{L:set low conditions b}).   Therefore, the right-hand side of (\ref{E:pi1 mod q3}) does not depend on the choice of $a$ and $b$.   So by making a suitable  initial choice of $\varepsilon \in \{\pm 1\}$, that does not depend on $a$ and $b$, we will have $\legendre{\pi_1}{\q_3}=-1$ and hence $\pi_1$ is not a square in $K_{\q_3}$.
\end{proof}

For the rest of the proof, we shall assume that $\varepsilon\in \{\pm 1\}$ has been chosen so that the conclusion of Lemma~\ref{L:pi1 not square mod q3} holds.

\begin{lemma} \label{L:q2 is inert}
The prime $\q_2$ is inert in $L$.
\end{lemma}
\begin{proof}
For $\p\in S_0$, we have $v_\p(a+b)=0$ by Lemma~\ref{L:set low conditions}(\ref{L:set low conditions c}).   For $\p\in S_3$, $(a,b)\in U_\p$ implies that $v_\p(a+b)=0$.  Since $(a+b)\OO_{K,S} = \q_2 \OO_{K,S}$, we have  
\[
(a+b)\OO_{K} = \q_2 \prod_{\p \in S} \p^{v_\p(a+b)} = \q_2 \prod_{\p \in S_1\cup S_2} \p^{v_\p(a+b)}.
\]
Taking their classes in $\Cl_K^{\n}$ and applying $\psi$ gives
\[
\psi([(a+b)\OO_K])= \psi([\q_2])   \prod_{\p \in S_1\cup S_2} \psi([\p])^{v_\p(a+b)}= \psi([\q_2]) \prod_{\p\in S_1,\, \p \text{ inert in $L$}} (-1)^{v_\p(a+b)},
\]
where the last equality uses that all the prime ideals in $S_2$ split in $L$.   Since $(a,b)\in U_\p$ for all $\p\in S_1$, Lemma~\ref{L:set low conditions}(\ref{L:set low conditions e}) implies that $\psi([\q_2])=-\psi([(a+b)\OO_K])$.

Since $(a,b)\in U_\p$ for all $\p \in S_0$, Lemma~\ref{L:set low conditions}(\ref{L:set low conditions d}) implies that $a+b\in 1+\n_0\OO_\p$ for all prime ideals $\p|\n_0$.  By our choice of $a$ and $b$, $a+b$ is positive in $K_v$ for all real place $v$ of $K$.   Therefore, we have $[(a+b)\OO_K]=1$ in $\Cl_K^{\n}$.  So $\psi([\q_2])=-1$ and hence $\q_2$ is inert in $L$.
\end{proof}

We now compute ratios of certain Selmer groups.

\begin{lemma} \label{L:Selmer ratio calculation}
We have 
\[
\frac{|\Sel_{\phi_1}(E_1/K)|}{|\Sel_{\hat\phi_1}(E'_1/K)|}=\frac{1}{2} \quad\text{ and }\quad \frac{|\Sel_{\phi_D}(E_D/K)|}{|\Sel_{\hat\phi_D}(E'_D/K)|}=1.
\]
\end{lemma}
\begin{proof}
Take any $d\in\{1,D\}$ and define $\tau_d:=|\Sel_{\phi_d}(E_d/K)|/|\Sel_{\hat\phi_d}(E'_d/K)|$.  By Lemma~\ref{L:Selmer ratio}(\ref{L:Selmer ratio i}), we have $\tau_d =  \prod_v \tfrac{1}{2} |\operatorname{Im}(\delta_{d,v})|$, where the product is over the places $v$ of $K$.   We have
\[
\tau_d = \prod_{v\in S\cup \{\q_1,\q_2,\q_3\} \cup S_\infty} \tfrac{1}{2} |\operatorname{Im}(\delta_{d,v})|
\]
by Lemma~\ref{L:Selmer}(\ref{L:Selmer i}) and (\ref{E:q}).

For $\p\in S$, $(a,b)\in U_\p$ implies that $\operatorname{Im}(\delta_{d,\p})=\Phi_{d,\p}$ by Lemma~\ref{L:Phidp validation}.  For each complex place $v$, we trivially have $\operatorname{Im}(\delta_{d,v})=1=\Phi_{d,v}$.  For each real place $v$, we have $0<a<b$ in $K_v$ by our choice of $a$ and $b$ and hence $2a(a+b)^2(a-b) d^2 <0$ in $K_v$.  For each real place $v$, Lemma~\ref{L:Selmer ratio}(\ref{L:Selmer ratio iii}) implies that $\operatorname{Im}(\delta_{d,v})=1=\Phi_{d,v}$.     For $\p\in \{\q_1,\q_2,\q_3\}$, we define $\Phi_{d,\p}:=\operatorname{Im}(\delta_{\p,d})$.

Therefore,
\[
\tau_d = \prod_{v\in S_0\cup S_1 \cup S_\infty} \tfrac{1}{2} |\Phi_{d,v}| \cdot \prod_{\p\in S_2 \cup S_3 \cup \{\q_1,\q_2,\q_3\} } \tfrac{1}{2}|\Phi_{d,\p}| =2^{-|S_2|} |V_d|  \prod_{\p\in S_2 \cup S_3 \cup \{\q_1,\q_2,\q_3\} } \tfrac{1}{2}|\Phi_{d,\p}| 
\]
where the last equality uses Lemma~\ref{L:prelim Selmer product}.   We have $|\Phi_{d,\p}|=4$ for $\p\in S_2$ and $|\Phi_{d,\p}|=1$ for $\p\in S_3$.   Therefore, 
\[
\tau_d = |V_d| 2^{-|S_3|} \prod_{\p\in \{\q_1,\q_2,\q_3\} } \tfrac{1}{2}|\Phi_{d,\p}| = \prod_{\p\in \{\q_1,\q_2,\q_3\} } \tfrac{1}{2}|\operatorname{Im}(\delta_{d,\p})|,
\]
where the last equality uses that $|V_d|=2^{n+1}$ and $|S_3|=n+1$. By (\ref{E:q}) and Lemma~\ref{L:Selmer}(\ref{L:Selmer iii}), we have $\tfrac{1}{2} |\operatorname{Im}(\delta_{d,\q_1})| =1$. By (\ref{E:q}) and Lemma~\ref{L:Selmer}(\ref{L:Selmer ii}), we have $\tfrac{1}{2} |\operatorname{Im}(\delta_{d,\q_3})| =1/2$.  Therefore, $\tau_d = |\operatorname{Im}(\delta_{d,\q_2})|/4$.

By (\ref{E:q}) and Lemma~\ref{L:Selmer}(\ref{L:Selmer iv}), we have $|\operatorname{Im}(\delta_{1,\q_2})|=2$ and hence $\tau_1=1/2$.  By Lemma~\ref{L:q2 is inert}, $\q_2$ is inert in $L$ and hence $D$ is not a square in $K_{\q_3}$.  By (\ref{E:q}) and Lemma~\ref{L:Selmer}(\ref{L:Selmer v}), we have $|\operatorname{Im}(\delta_{D,\q_2})|=4$ and hence $\tau_D=1$.
\end{proof}

We next compute the Selmer groups associated to the isogenies $\phi_1$ and $\phi_D$.

\begin{lemma} \label{L:first Selmer}
For $d\in \{1,D\}$, we have $\Sel_{\phi_d}(E_d/K)=\{1, 2a(a+b)\cdot (K^\times)^2\}$.
\end{lemma}
\begin{proof}
Take any $d\in \{1,D\}$.  From the point $(0,0)\in E'_d(K)$, we find that $2a(a+b)\cdot (K^\times)^2$ is an element of $\Sel_{\phi_d}(E_d/K) \subseteq K^\times/(K^\times)^2$.  Now take any $\alpha \in \Sel_{\phi_d}(E_d/K)$.   We need to show that $\alpha \in \{1, 2a(a+b)\cdot (K^\times)^2\}$.  

Choose an element $c \in K^\times$ in the square class $\alpha$.  For each $\p\in S$, $c\cdot (K_\p^\times)^2$ lies in $\operatorname{Im}(\delta_{d,\p})$ since $\alpha$ is an element of $\Sel_{\phi_d}(E_d/K)$.  For each $\p\in S$, we have $(a,b)\in U_\p$ and hence $c\cdot (K_\p^\times)^2$ is an element of $\Phi_{d,\p}$ by Lemma~\ref{L:Phidp validation}.

Take any nonzero prime ideal $\p\notin S_0\cup S_1 \cup S_2 \cup\{\q_1,\q_2\}$ of $\OO_K$.   If $\p \in S_3$, then $c$ is a square in $K_\p$ since $\Phi_{d,\p}=1$.  If $\p=\q_3$, then $c$ is a square in $K_\p$ by (\ref{E:q}) and Lemma~\ref{L:Selmer}(\ref{L:Selmer ii}).  If also $\p\notin S_3\cup \{\q_3\}$, then (\ref{E:q}) implies that $v_\p(a)=v_\p(a+b)=v_\p(a-b)=0$ and hence $c\cdot (K_\p^\times)^2$ is represented by an element of $\OO_\p^\times$ by Lemma~\ref{L:Selmer}(\ref{L:Selmer i}).  Therefore, the congruence
\[
v_\p(c)\equiv 0 \pmod{2}
\] 
holds for all nonzero prime ideals $\p \notin S_0\cup S_1\cup \{\q_1,\q_2\}$ of $\OO_K$.  Since $\OO_{K,S_0}$ is a PID by our choice of $S_0$, we may assume that $c\in \alpha$ was chosen so that $v_\p(c)=0$ for all $\p\notin S_0\cup S_1 \cup S_2 \cup\{\q_1,\q_2\}$.  We also have $v_\p(a)=v_\p(a+b)=0$ for all $\p\notin S_0\cup S_1 \cup S_2 \cup\{\q_1,\q_2\}$ because of (\ref{E:q}) and since $(a,b)\in U_\p$ for $\p\in S_3$.   Since $v_{\q_1}(a)=1$, $v_{\q_2}(a)=0$, $v_{\q_1}(a+b)=0$ and $v_{\q_2}(a+b)=1$, we find that 
\[
c = u a^e (a+b)^{e'}
\]
for unique integers $e,e'\in \ZZ$ and a unique $u\in \OO_{K,S_0\cup S_1\cup S_2}^\times$.  After multiplying $\alpha$ by an appropriate power of $2a(a+b)\cdot (K^\times)^2 \in \Sel_{\phi_d}(E_d/K)$ and using that $2$ is a unit in $\OO_{K,S_0\cup S_1\cup S_2}$, we may assume without loss of generality that  $c=u a^e$ with $u \in \OO_{K,S_0\cup S_1 \cup S_2}^\times$ and $e\in \ZZ$.    It suffices to show that $c$ is a square in $K^\times$.

Consider the prime $\p_0\in S_3$.   We know that $c=ua^e$ is a square in $K_{\p_0}$ since $\Phi_{d,\p_0}=1$. Since $u$ is a square in $K_{\p_0}$ by Lemma~\ref{L:p0 existence}, we deduce that $a^e$ is a square in $K_{\p_0}$.  However, $(a,b)\in U_{\p_0}$ implies that $a$ is not a square in $K_{\p_0}$.  Therefore, $e$ must be even.  After replacing $c$ by itself times a suitable square in $K^\times$, we may assume that $c \in \OO_{K,S_0\cup S_1\cup S_2}^\times$.   

For each complex place $v$, we trivially have $\operatorname{Im}(\delta_{d,v})=1=\Phi_{d,v}$.  For each real place $v$, we have $0<a<b$ in $K_v$ by our choice of $a$ and $b$ and hence $2a(a+b)^2(a-b) d^2 <0$ in $K_v$.  For each real place $v$, Lemma~\ref{L:Selmer ratio}(\ref{L:Selmer ratio iii}) then implies that $\operatorname{Im}(\delta_{d,v})=1=\Phi_{d,v}$.   Therefore, $c \cdot (K_v^\times)^2$ is an element of $\Phi_{d,v}$ for all $v\in S_0\cup S_1\cup S_\infty$.   In particular, $c \cdot (\OO_{K,S_0\cup S_1\cup S_2}^\times)^2$ lies in $\ker \gamma_d := V_d$.

For every $\p\in S_3$, $\Phi_{d,\p}=1$ and hence $c$ is a square in $K_\p$.   Since $c \cdot (\OO_{K,S_0\cup S_1\cup S_2}^\times)^2$ is in $V_d$ and $c \cdot (K_\p^\times)^2=(K_\p^\times)^2$ for all $\p\in \{\p_1,\ldots, \p_n\}\subseteq S_3$, Lemma~\ref{L:S3} implies that $c \cdot (K^\times)^2$ is in the group generated by $\pi_1\cdot (K^\times)^2$.  After replacing $c$ by itself times a suitable square in $K^\times$, we may assume that $c=1$ or $c=\pi_1$.

By Lemma~\ref{L:Selmer}(\ref{L:Selmer ii}), we have $\operatorname{Im}(\delta_{d,\q_3})=1$ and hence $c$ is a square in $K_{\q_3}$.   By our choice of $\varepsilon\in \{\pm 1\}$, Lemma~\ref{L:pi1 not square mod q3} implies that $\pi_1$ is not a square in $K_{\q_3}$.  Therefore, $c=1$ which completes the proof.
\end{proof}

\begin{lemma} \label{L:WMW}
\begin{romanenum}
\item \label{L:WMW i}
For $d\in \{1,D\}$, the group $E'_d(K)/\phi_d(E_d(K))$ is isomorphic to $\ZZ/2\ZZ$ and is generated by $(0,0)$.

\item \label{L:WMW ii}
The group $E_D(K)/\hat\phi_D(E'_D(K))$ is isomorphic to $\ZZ/2\ZZ$ and is generated by $(0,0)$.

\item \label{L:WMW iii}
The group $E_1(K)/\hat\phi_1(E'_1(K))$ is isomorphic to $(\ZZ/2\ZZ)^2$ and is generated by $P_0:=(0,0)$ and $P_1:=(-2a(a+b),2a(a+b)^2)$.
\end{romanenum}
\end{lemma}
\begin{proof}
Take $d\in \{1,D\}$.  From \S\ref{SS:2-descent}, there is an injective homomorphism 
\[
\delta_d\colon E'_d(K)/\phi_d(E_d(K))\hookrightarrow \Sel_{\phi_d}(E_d/K)\subseteq K^\times/(K^\times)^2
\] 
for which $\delta_d((0,0))=2a(a+b)\cdot (K^\times)^2$.   Part (\ref{L:WMW i}) is now immediate from Lemma~\ref{L:first Selmer}; note that $2a(a+b)$ is not a square in $K^\times$ since its valuation with respect to the prime ideal $\q_1$ is $1$.    In particular, $|\Sel_{\phi_d}(E_d/K)|=2$.

By Lemma~\ref{L:Selmer ratio calculation}, we have $|\Sel_{\hat\phi_1}(E'_1/K)|=4$ and $|\Sel_{\hat\phi_D}(E'_D/K)|=2$.  From \S\ref{SS:2-descent}, there is an injective homomorphism 
\begin{align} \label{E:final delta}
\delta_d'\colon E_d(K)/\hat\phi_d(E'_d(K))\hookrightarrow \Sel_{\hat\phi}(E'_d/K) 
\end{align} 
that satisfies $\delta_d'((0,0))=2a(a-b)\cdot (K^\times)^2$.    Note that $2a(a-b)$ is not a square in $K^\times$ since its valuation with respect to $\q_1$ is $1$.  When $d=D$, (\ref{E:final delta}) is an isomorphism since $|\Sel_{\hat\phi}(E'_D/K)|=2$.   This completes the proof of (\ref{L:WMW ii}).

Finally consider $d=1$.  Note that $P_0$ and $P_1$ are points in $E_1(K)$.    By considering the valuations with respect to the prime ideals $\q_2$ and $\q_3$, we find that $\delta'_1(P_0)=2a(a-b)\cdot (K^\times)^2$ and $\delta'_1(P_1)=-2a(a+b) \cdot (K^\times)^2$ generate a subgroup of $\Sel_{\hat\phi}(E'/K)$ that has cardinality $4$.    Since $\Sel_{\hat\phi_1}(E'_1/K)$ has cardinality $4$, we deduce that (\ref{E:final delta}) is an isomorphism when $d=1$ and that the group $E_1(K)/\hat\phi_1(E'_1(K))$ is generated by $P_0$ and $P_1$.
\end{proof}

To finish the proof, we will show that $E_1$ has rank $1$ and $E_D$ has rank $0$.  Take any $d\in \{1,D\}$.  We have $E_d(K)[2]=\{0,(0,0)\}$ since the discriminant of $x^2+4a(a+b)dx+2a(a+b)^2(a-b)d^2$ is $8a(a+b)^3d^2$ which is not a square in $K$; it has valuation $1$ with respect to the prime ideal $\q_1$.   We have $\phi_d(E_d(K)[2]) =\phi_d(\langle (0,0) \rangle)=0$.   
By the exact sequence (\ref{E:exact MW}) and Lemma~\ref{L:WMW}(\ref{L:WMW i}), we find that the quotient map
\[
E_d(K)/2E_d(K) \to E_d(K)/\hat{\phi}_d(E'_d(K))
\]
is an isomorphism of groups. By Lemma~\ref{L:WMW}(\ref{L:WMW ii}) and (\ref{L:WMW iii}), we obtain isomorphisms 
\begin{align} \label{E:hard weak MW}
E_1(K)/2E_1(K)  \cong (\ZZ/2\ZZ)^2 \quad \text{and}\quad E_D(K)/2E_D(K)  \cong \ZZ/2\ZZ.
\end{align}
Since $E_d(K)$ is a finitely generated abelian group with $E_d(K)[2]\cong \ZZ/2\ZZ$, we have 
\begin{align} \label{E:easy weak MW}
E_d(K)/2E_d(K)\cong (\ZZ/2\ZZ)^{r_d+1},  
\end{align}
where $r_d$ is the rank of $E_d(K)$.   From (\ref{E:hard weak MW}) and (\ref{E:easy weak MW}), we conclude that $r_1=1$ and $r_D=0$.

\section{Proof of Theorems~\ref{T:cor} and \ref{T:main}} \label{S:secondary proofs}

We need only proof Theorem~\ref{T:main} since it clearly implies Theorem~\ref{T:cor}.    Let $L/K$ be any quadratic extension of number fields.   Choose a $D\in K^\times$ for which $L=K(\sqrt{D})$.
With this $K$ and $D$, we let $S$ and $\{U_\p\}_{\p \in S}$ be as in Theorem~\ref{T:true}.  We now apply Kai's result.

\begin{lemma} \label{L:Kai}
There are nonzero $a,b\in \OO_{K,S}$ such that $a$, $a+b$ and $a-b$ generate distinct nonzero prime ideals of $\OO_{K,S}$, $(a,b)$ lies in $U_\p$ for all $\p \in S$, and $0<a<b$ in $K_v$ for all real places $v$ of $K$.
\end{lemma}
\begin{proof}
For each $\p\in S$, choose a pair $(a_\p,b_\p) \in U_\p$.   Since the sets $U_\p$ are open in $K_\p^2$ and $S$ is finite, we can choose $0<\epsilon < 1/2$ such that $\{(x,y) \in K_\p^2 : |x-a_\p|_\p<\epsilon,\, |y-b_\p|_\p <\epsilon\} \subseteq U_\p$ for all $\p \in S$.

Using Proposition~13.2 of \cite{Kai} we find that are nonzero $a,b\in \OO_{K,S}$ such that the following hold:
\begin{itemize}
\item
$|a-a_\p|_\p<\epsilon$ and $|b-b_\p|_\p <\epsilon$ for all $\p\in S$,
\item
$|a/b - 1/2|_v < \epsilon$ for all real places $v$ of $K$,
\item 
$b>0$ in $K_v$ for all real places $v$ of $K$,
\item
$a$, $a+b$ and $a-b$ generate distinct prime ideals of $\OO_{K,S}$.
\end{itemize}
From our choice of $\epsilon$, we have $(a,b)\in U_\p$ for all $\p\in S$.   Finally take any real place $v$ of $K_v$.  We have $|a/b-1/2|_v<\epsilon<1/2$ and hence $0<a/b< 1$ in $K_v$.  Since $b>0$ in $K_v$, we have $0<a<b$ in $K_v$.
\end{proof}

Take any $a,b\in \OO_{K,S}$ as in Lemma~\ref{L:Kai}.   By Theorem~\ref{T:true}, the elliptic curve $E_1$ over $K$ defined by (\ref{E:intro strategy weierstrass}) has rank $1$ and its quadratic twist $E_D$ by $D$ has rank $0$.  Since $L=K(\sqrt{D})$ and $E_D$ is the quadratic twist of $E_1$ by $D$, we have
\[
\rank \, E_1(L) = \rank\, E_1(K)  + \rank\, E_D(K);
\]
this can be proved by using the action of $\Gal(L/K)$ on $E(L)\otimes_\ZZ\QQ$ and considering the eigenspaces of the nonidentity element of $\Gal(L/K)$.   Thus $E_1(L)$ has rank $1+0=1$.  Therefore, $E_1$ is an elliptic curve over $K$ for which $E_1(K)$ and $E_1(L)$ both have rank $1$.  It remains to show that there are infinitely many such elliptic curves up to isomorphism.

Let $j\in K$ be the $j$-invariant of $E_1$.   We have
\[
j= 64\frac{(5a+3b)^3}{(a-b)^2(a+b)}.
\]
For each $\p\in S$, we let $U_\p'$ be the set of $(x,y)\in U_\p$ for which $64(5x+3y)^3 - j (x-y)^2(x+y)\neq 0$.   Since $U_\p$ is a nonempty open subset of $K_\p^2$ so is $U_\p'$.  Of course Theorem~\ref{T:true} will remain true if we replace each $U_\p$ by the smaller nonempty open set $U_\p'$.  By repeating the above argument with the smaller sets $U_\p'$, we will obtain an elliptic curve $E/K$ with $\rank\, E(K)=\rank E(L) = 1$ whose $j$-invariant is not equal to $j$.  Note that $E_1$ and $E$ are not isomorphic since they have different $j$-invariants.  By repeating this process of shrinking $U_\p$ over and over to exclude certain $j$-invariants, we deduce that there are infinitely many elliptic curves $E$ over $K$ for which $E(K)$ and $E(L)$ both have rank $1$.

\end{document}